%% file: Inertial_Convex_Bilevel_Optimization_V3_4.tex
\documentclass[
11pt, 
a4paper, 
oneside, 
headinclude,footinclude, 
BCOR5mm, 
]{scrartcl}

\input{structure.tex} 
\usepackage{titlesec}
\titlelabel{\thetitle.\;}
\usepackage[utf8]{inputenc}
\usepackage[english]{babel}
\usepackage{enumitem}
\usepackage{amsmath}
\usepackage{amssymb}
\usepackage{amsthm}
\usepackage{amsfonts}
\usepackage{epsfig}
\PassOptionsToPackage{hyphens}{url}\usepackage{hyperref}
\hypersetup{
    colorlinks=true,
    linkcolor=blue,
    filecolor=magenta,
    urlcolor=cyan,
}
\usepackage{caption}
\numberwithin{equation}{section}
\newtheorem{dfn}{Definition}[section]
\newtheorem{lem}[dfn]{Lemma}
\newtheorem{thm}[dfn]{Theorem}

\newtheorem{prop}[dfn]{Proposition}

\theoremstyle{definition}

\newtheorem{asm}[dfn]{Assumption}
\newtheorem{exm}[dfn]{Example}
\newtheorem{rem}[dfn]{Remark}

\usepackage{algorithm}
\usepackage{algorithmic}
\makeatletter
\def\listofalgorithms{\@starttoc{loa}\listalgorithmname}
\def\l@algorithm{\@tocline{0}{3pt plus2pt}{0pt}{1.9em}{}}
\renewcommand{\ALG@name}{Algorithm}
\renewcommand{\listalgorithmname}{List of \ALG@name s}
\numberwithin{algorithm}{section}
\makeatletter
\DeclareOldFontCommand{\rm}{\normalfont\rmfamily}{\mathrm}
\DeclareOldFontCommand{\sf}{\normalfont\sffamily}{\mathsf}
\DeclareOldFontCommand{\tt}{\normalfont\ttfamily}{\mathtt}
\DeclareOldFontCommand{\bf}{\normalfont\bfseries}{\mathbf}
\DeclareOldFontCommand{\it}{\normalfont\itshape}{\mathit}
\DeclareOldFontCommand{\sl}{\normalfont\slshape}{\@nomath\sl}
\DeclareOldFontCommand{\sc}{\normalfont\scshape}{\@nomath\sc}
\makeatother
\makeatletter
\def\@seccntformat#1{\@ifundefined{#1@cntformat}%
   {\csname the#1\endcsname\space}
   {\csname #1@cntformat\endcsname}}
\newcommand\section@cntformat{\thesection.\space}       
\newcommand\subsection@cntformat{\thesubsection.\space} 
\makeatother
\allowdisplaybreaks
\title{\normalfont\spacedlowsmallcaps{An inertial extrapolation method}\\ \normalfont\spacedlowsmallcaps{for convex simple bilevel optimization}} 

\author{Yekini Shehu$^{\dag}$,  Phan Tu Vuong$^{\ddag}$, and  Alain Zemkoho$^{\natural}$}

\date{\today} 

\titleformat{\subsection}[runin]
  {\normalfont\large\bfseries}{\thesubsection}{0.5em}{}
\begin{document}


\renewcommand{\sectionmark}[1]{\markright{\spacedlowsmallcaps{#1}}} 
\lehead{\mbox{\llap{\small\thepage\kern1em\color{halfgray} \vline}\color{halfgray}\hspace{0.5em}\rightmark\hfil}} 

\pagestyle{scrheadings} 


\maketitle 

\setcounter{tocdepth}{2} 

%
%


\section*{Abstract} 
\noindent We consider a scalar objective minimization problem over the solution set of another optimization problem. This problem is known as \emph{simple bilevel optimization problem} and has drawn a significant attention in the last few years. Our inner problem consists of minimizing the sum of smooth and nonsmooth functions while the outer one is the minimization of a smooth convex function. We propose and establish the convergence of a fixed-point iterative method with inertial extrapolation to solve the problem. Our numerical experiments show that the method proposed in this paper outperforms the currently best known algorithm to solve the class of problem considered.



\let\thefootnote\relax\footnotetext{$\dag$ \textit{Department of Mathematics, University of Nigeria, Nsukka, Nigeria; e-mail: \url{yekini.shehu@unn.edu.ng}. Current address (May 2016 -- April 2019): Institute of Mathematics, University of W\"urzburg, Emil-Fischer-Str.\ 30,
97074 W\"urzburg, Germany. The research of this author is supported by the
Alexander von Humboldt-Foundation.}}

\let\thefootnote\relax\footnotetext{$\ddag$ \textit{Faculty of Mathematics, University of Vienna, Oskar-Morgenstern-Platz 1, 1090 Vienna, Austria; e-mail:
\url{vuong.phan@univie.ac.at}. The work of this author is funded by the FWF Grant M 2499 Meitner-Programm.}}

\let\thefootnote\relax\footnotetext{$\natural$ \textit{School of Mathematics, University of Southampton
       SO17 1BJ Southampton, UK; e-mail: \url{a.b.zemkoho@soton.ac.uk}. The work of this author is funded by the  EPSRC Grant EP/P022553/1.}}


\section{Introduction}
Our main aim in this paper is to solve a scalar objective minimization problem over the solution set of another optimization problem; i.e., precisely, the problem
\begin{eqnarray}\label{bilevel1}
\min~h(x) \; \mbox{ s.t. } \; x \in X^*\subseteq \mathbb{R}^n,
\end{eqnarray}
where $h : \mathbb{R}^n \rightarrow \mathbb{R}$ is assumed to be strongly convex and differentiable, while $X^*$ is the nonempty set of minimizers of the classical convex composite optimization problem
\begin{eqnarray}\label{bilevel2}
\min~\varphi(x) := f(x) + g(x),
\end{eqnarray}
where $f : \mathbb{R}^n \rightarrow \mathbb{R}$ is continuously differentiable and $g$, an extended real-valued function on $\mathbb{R}^n$, which can be nonsmooth. Problem \eqref{bilevel1}--\eqref{bilevel2} was labeled in \cite{DempeDinhDutta2010} as \emph{simple bilevel optimization problem}, as opposed to the more general version of the problem (see, e.g., \cite{DempeFoundations}), where the follower's problem \eqref{bilevel2} is parametric, with the parameter representing the variable controlled by the leader, which is in turn different from the one under the control of the follower. For more details on the vocabulary and connections of problem  \eqref{bilevel1}--\eqref{bilevel2} to the standard bilevel optimization problem, see Subsection \ref{Standard bilevel optimization} below. \\

\noindent A common approach to solve problem \eqref{bilevel1}--\eqref{bilevel2} consists of the
Tikhonov-type regularization \cite{Tikhonov} (indirect method), based on solving the following
regularized problem
\begin{eqnarray}\label{bilevel3}
\min~\varphi_\lambda(x) := \varphi(x) + \lambda h(x)
\end{eqnarray}
for some $\lambda >0$. Note that problem  \eqref{bilevel1}--\eqref{bilevel2} can be traced back to the work by Mangasarian and Meyer \cite{MangasarianMeyer1979} in the process of developing efficient algorithms for large scale linear programs. The model emerged in turn as a refinement of the regularization technique introduced by Tikhonov \cite{Tikhonov}. The underlying idea in the related papers by Mangasarian and his co-authors is called \emph{finite-perturbation property}, which consists of finding a parameter $\bar \lambda$ (\emph{Tikhonov perturbation parameter}) such that for all $\lambda\in [0, \; \bar \lambda]$,
\begin{eqnarray}\label{Finite-perturbation}
\arg\underset{x \in X^*}\min~h(x) = \arg\underset{x \in \mathbb{R}^n}\min~\varphi_\lambda(x):=\varphi(x) + \lambda h(x).
\end{eqnarray}
This property, initially proven in \cite{MangasarianMeyer1979} when the lower-level problem is a linear program, was later extended in \cite{FerrisMangasarian} to the case where it is a general convex optimization problem. \\

\noindent To the best of our knowledge, the development of solution algorithms specifically tailored to optimization problems of the form \eqref{bilevel1}--\eqref{bilevel2} can be traced back to the work by Cabot \cite{Cabot2005}, where a \emph{proximal point method} is proposed to solve the problem and its extension to a simple hierarchical optimization problem with finitely many levels. In contrary to the latter paper, where the approximation scheme is only implicit thus making the method not easy to numerically implement,  Solodov \cite{Solodov2} proposed an explicit and more tractable proximal point method for problem \eqref{bilevel1}--\eqref{bilevel2}. Since then, various proximal point algorithms have been developed to solve the problems under different types of frameworks, see, e.g., \cite{BotNguyen2018, malitsky2017chambolle, SabachShtern} and references therein.\\

%
\noindent Motivated by the results in \cite{BeckSabach}, Sabach and Shtern  \cite{SabachShtern} recently proposed
the following scheme (with $x_0 \in \mathbb{R}^n$ as starting point), called \emph{Bilevel Gradient Sequential Averaging Method} (abbreviated as \emph{BiG-SAM}), to solve problem \eqref{bilevel1}--\eqref{bilevel2}:
   \begin{equation}\label{e1}
\left\{\begin{array}{l}
s_n={\rm prox}_{\lambda g}(x_{n-1}-\lambda \nabla f(x_{n-1}))\\
z_n=x_{n-1}-\gamma \nabla h(x_{n-1})\\
x_{n+1}=\alpha_n z_n+(1-\alpha_n)s_n,~~n \geq 1
\end{array}\right.
\end{equation}
with $\lambda \in \left(0,\frac{1}{L_f}\right]$, $\gamma \in \left(0,\frac{2}{L_h +\sigma}\right]$, and $\{\alpha_n\}$ satisfying the conditions assumed in \cite{XuViscosity}. Sabach and Shtern \cite{SabachShtern} obtained a nonasymptotic O($\frac{1}{n}$) global rate of convergence in terms
of the inner objective function values and showed that BiG-SAM \eqref{e1} appears simpler and cheaper than the method proposed in \cite{BeckSabach}.
The numerical example in \cite{SabachShtern} also showed that BiG-SAM \eqref{e1} outperforms the method in \cite{BeckSabach} for solving problem \eqref{bilevel1}--\eqref{bilevel2}. The algorithm in \cite{SabachShtern} seems to be the most efficient method developed so far for convex simple bilevel optimization problems.\\


\noindent Inspired by recent results on \emph{inertial extrapolation} type algorithms for solving optimization problem
(see, e.g., {\cite{Attouch3,Beck,Bot2,Ochs} and references therein),
 our aim in this paper is to solve problem \eqref{bilevel1}--\eqref{bilevel2} by introducing an inertial extrapolation step to BiG-SAM \eqref{e1}
 (which we shall call \emph{iBiG-SAM}). 
  We then establish the global convergence of our method  under reasonable assumptions. Numerical experiments show that the proposed method outperforms the BiG-SAM \eqref{e1} introduced in \cite{SabachShtern}.\\ 

\noindent For the remainder of the paper, first note that there is a striking similarity between the exact penalization model \eqref{bilevel3}   and a corresponding partial penalization approach based on the \emph{partial calmness concept} \cite{YeZhuOptCondForBilevel1995} often used to solve the general bilevel optimization problem. Both approaches seem to have originated from completely different sources and their development also seems to be occurring  independently from each other till now. In Subsection \ref{Standard bilevel optimization}, we clarify this similarity and discuss some strong relationships between the two problem classes. In Subsection \ref{Sec:Prelims}, we recall some basic definitions and results that will play an important role in the paper. The proposed method and its convergence analysis are presented in Section~\ref{Sec:Method}. Some numerical experiments are given in Section~\ref{Sec:Numerics}. We conclude the paper with some final remarks in Section~\ref{Sec:Final}.

\section{General context and mathematical tools}
\subsection{Standard bilevel optimization.} \label{Standard bilevel optimization} 
 In this subsection, we provide a discussion to place the simple bilevel optimization introduced above in a general context of bilevel optimization. To proceed, we consider a simple optimistic version of the latter class of problem, which aligns suitably with problem \eqref{bilevel1}--\eqref{bilevel2}, i.e.,
\begin{equation}\label{Standard}
    \underset{x,y}\min~h(x, y) \;\, \mbox{ s.t. }\;\, y\in S(x)
\end{equation}
where $h : \mathbb{R}^n\times \mathbb{R}^m \rightarrow \mathbb{R}$ represents the upper level objective function and the set-valued mapping $S$ defines the the set of optimal solutions of the lower level problem
\begin{equation}\label{lower-level}
    \underset{y}\min~\varphi(x,y)
\end{equation}
($\varphi : \mathbb{R}^n\times \mathbb{R}^m \rightarrow \mathbb{R}$) for any fixed upper level variable $x$. Obviously, problem \eqref{bilevel1}--\eqref{bilevel2} is a special case of problem \eqref{Standard}--\eqref{lower-level}, where the optimal solution of the leader is simply picked among the optimal solutions of the lower level problem, which in turn are obtained without any influence from the leader as it is the case in the latter problem.\\

\noindent On the other hand problem \eqref{Standard}--\eqref{lower-level} can be equivalently written as the following optimization problem over an efficient set
$$
\underset{x,y}\min~h(x,y) \; \mbox{ s.t. }\; (x,y)\in E\left(\mathbb{R}^n\times \mathbb{R}^m, \; \bar{\varphi},\; \preccurlyeq\right),
$$
where $E\left(\mathbb{R}^n\times \mathbb{R}^m, \; \bar{\varphi},\; \preccurlyeq\right)$ denotes the efficient set (i.e., optimal solution set) of the problem of minimizing a multiobjective function $\bar \varphi$ (based on $\varphi$ \eqref{lower-level}) over $\mathbb{R}^n\times \mathbb{R}^m$ w.r.t. a certain order relation $\preccurlyeq$; for examples of choices of the latter function and corresponding order relations, see the papers \cite{Eichfelder, Fulop}. Obviously, an optimization problem over an efficient set is a generalization of the simple bilevel optimization problem \eqref{bilevel1}--\eqref{bilevel2}, and has been extensively investigated since the seminal work by Philip \cite{Phillip1972}; see \cite{Yamamoto2002} for a literature review on the topic.\\

\noindent One common approach to transform problem \eqref{Standard}--\eqref{lower-level} into a single-level optimization problem is the so-called lower-level optimal value function (LLVF) reformulation
\begin{equation}\label{LLVF}
    \underset{x, y}\min~h(x, y) \;\; \mbox{ s.t. }\;\; \varphi(x,y)\leq \varphi^*(x),
\end{equation}
where the function $\varphi^*(x)=\underset{y}\min~\varphi(x,y)$ represents the optimal value function of the lower level problem \eqref{lower-level}.
Recall that this reformulation is an underlying feature in the development of the link \eqref{Finite-perturbation} between the simple bilevel optimization problem \eqref{bilevel1}--\eqref{bilevel2} and the penalized problem \eqref{bilevel3} as outlined in the corresponding publications; see, e.g., \cite{FerrisMangasarian}. However, we instead want to point out here an interesting similarity between the finite termination property \eqref{Finite-perturbation} and the \emph{partial calmness} concept \cite{YeZhuOptCondForBilevel1995} commonly used in the context of standard bilevel optimization. To highlight this, let $(\bar x, \bar y)$ be a local optimal solution of \eqref{bilevel1}--\eqref{bilevel2}. The problem is partially calm at $(\bar x, \bar y)$ if and only if there exists $\lambda >0$ such that $(\bar x, \bar y)$ is also a local optimal solution of the penalized problem
\begin{equation}\label{Penalized LLVF}
    \underset{x, y}\min~h(x, y) + \lambda\left(\varphi(x,y)-\varphi^*(x)\right).
\end{equation}
The partial calmness concept does not automatically hold for the simple bilevel optimization problem \eqref{bilevel1}--\eqref{bilevel2}. To see this, consider the example of convex simple bilevel optimization problem of minimizing $(x-1)^2$ subject to $x\in \arg\min~y^2$.
It is clear that  $0$ is the only optimal solution of this problem. But for the corresponding penalized problem \eqref{Penalized LLVF} to minimize $(x-1)^2 + \lambda x^2$,
we can easily check that the optimal solution is the number $x(\lambda):= \frac{1}{1+\lambda}$ for all $\lambda >0$. Clearly, $x(\lambda)\neq 0$ for all $\lambda >0$. \\

\noindent It is also important to note that, possibly unlike the finite termination property \eqref{Finite-perturbation}, the partial calmness concept was introduced as a qualification condition to derive necessary optimality conditions for problem \eqref{LLVF}; see \cite{DempeZemkohoGenMFCQ,YeZhuOptCondForBilevel1995} for some papers where this concept is used, and also the papers \cite{DempeDinhDuttaPandit2018, FrankeMehlitzPilecka2018} for new results on simple bilevel optimization problems from the perspective of standard bilevel optimization.

\subsection{Basic mathematical tools}\label{Sec:Prelims}
We state the following well-known lemmas which will be used in our convergence analysis in the sequel.

\begin{lem}\label{lm2}
The following  well-known results hold in $\mathbb{R}^n$:
\begin{itemize}
\item[(i)] $||x+y||^2=||x||^2+2\langle x,y\rangle+||y||^2, \;\;\forall x, y \in \mathbb{R}^n;$
\item[(ii)] $||x+y||^2 \leq ||x||^2+ 2\langle y, x+y \rangle, \;\;\forall x, y \in \mathbb{R}^n;$
\item[(iii)] $\|tx+sy\|^2=t(t+s)\|x\|^2+s(t+s)\|y\|^2-st\|x-y\|^2,\;\; \forall x, y \in \mathbb{R}^n, \;\; s, t \in \mathbb{R}.$
\end{itemize}
\end{lem}

\begin{lem}(see, e.g., \cite{xu2})\label{lm23}
Let $\{a_n\}$ and $\{\gamma_n\}$ be sequences  of nonnegative real numbers, $\{\alpha_n\}$ a sequence in (0,1) and $\{\sigma_n\}$ a real sequence satisfying the following relation:
$$a_{n+1}\leq(1-\alpha_n)a_n+\sigma_n+\gamma_n,~~n \geq 1.$$
\noindent Assume $\sum\gamma_n<\infty.$
Then the following results hold:
\begin{itemize}
\item[(i)] If $\sigma_n \leq \alpha_n M$ for some $M\geq 0$, then $\{a_n\}$ is a bounded sequence.
\item[(ii)] If $\sum\alpha_n=\infty$ and $\limsup \frac{\sigma_n}{\alpha_n}\leq 0$, then $\lim a_n=0$.\\
\end{itemize}
\end{lem}

\noindent We state the formal definition of some classes of operators that
play an essential role in our analysis in the sequel.

\begin{dfn}\label{Def:LipMon}
An operator $ T: \mathbb{R}^n\rightarrow \mathbb{R}^n $
is called
\begin{itemize}
   \item[(a)] {\em nonexpansive} if and only if $  \| Tx - Ty \| \leq \| x - y \| $ for all $ x, y \in \mathbb{R}^n $;
\item[(b)] {\em averaged} if and only if it can be written as
the average of the identity mapping $I$ and a nonexpansive operator, i.e.,
$T: = (1-\beta)I + \beta S$ with $\beta\in (0,1)$ and $S:\mathbb{R}^n\rightarrow \mathbb{R}^n$ being a nonexpansive operator. More precisely, we say that $T$ is $\beta$-averaged;
 \item[(c)] {\em firmly nonexpansive} if and only if $2T-I$ is nonexpansive, or equivalently,
$$
\langle Tx-Ty, x-y\rangle \geq \|Tx-Ty\|^2, \;\; \forall x, y \in \mathbb{R}^n.
$$
\noindent Alternatively, $T$ is said to be firmly nonexpansive if and only if it can be expressed as
$T:=\frac{1}{2}(I+S)$, where $S:\mathbb{R}^n\rightarrow \mathbb{R}^n$ is nonexpansive.
\end{itemize}
\end{dfn}

\noindent
We can see from above that firmly nonexpansive operators (in particular,
projections) are $\frac{1}{2}$-averaged.

\begin{lem} (\cite{Goebelkirk}) \label{lm25}
Let $T:\mathbb{R}^n\rightarrow \mathbb{R}^n$ be a nonexpansive operator. Let $\{x_n\}$ be a sequence in $\mathbb{R}^n$ and $x$ be a point in $\mathbb{R}^n$.
Suppose that $x_n\rightarrow x$ as $n\rightarrow \infty$ and that
$x_n-Tx_n\rightarrow 0$ as $n\rightarrow \infty$. Then, $x\in F(T)$, where $F(T)$ is the set of fixed points of $T$.\\
\end{lem}

\noindent Next, we provide some relevant properties of averaged operators.
\begin{prop}(see, e.g., \cite{cby})\label{need1}
For given operators $S$, $T$, and $V$ defined from $\mathbb{R}^n$ to $\mathbb{R}^n$, the following statements are satisfied:
\begin{itemize}
  \item[(a)] If $T = (1-\alpha)S + \alpha V$ for some $\alpha \in (0, 1)$ and if $S$ is averaged and $V$ is nonexpansive, then the operator $T$ is averaged.
\item[(b)] The operator $T$ is firmly nonexpansive if and only if the complement $I- T$ is also firmly nonexpansive.
\item[(c)] If $T = (1-\alpha)S + \alpha V$ for some $\alpha \in (0, 1)$ and if $S$ is firmly nonexpansive and $V$ is nonexpansive, then $T$ is averaged.
\item[(d)] The composite of finitely many averaged operators is averaged. That is, if for each $i=1, \ldots, N$, the operator $T_i$ is averaged, then so is the composite operator $T_1\ldots T_N$. In particular,
if $T_1$ is $\alpha_1$-averaged and $T_2$ is $\alpha_2$-averaged, where $\alpha_1,\alpha_2 \in (0, 1)$, then
the composite $T_1T_2$ is $\alpha$-averaged, where $\alpha = \alpha_1 +\alpha_2-\alpha_1\alpha_2.$\\
\end{itemize}
\end{prop}

\noindent Finally, for the last proposition of this section, we recall the definition of monotonicity of nonlinear operators.

\begin{dfn}\label{Def:LipMon}
Given is a nonlinear operator
$A$ with domain $D(A)$ in $\mathbb{R}^n$ and $\beta, \nu$ are positive constants. Then $A$ is called
\begin{itemize}
\item[(a)] {\em monotone} on $D(A)$ if $\langle Ax - Ay, x-y \rangle\geq 0$ for all $x,y \in D(A)$;
\item[(b)] {\em $\beta$-strongly monotone} if $\langle Ax - Ay, x-y \rangle \geq \beta\|x-y\|^2$ for all $x,y \in D(A)$;
\item[(c)] {\em $\nu$-inverse strongly monotone } ($\nu$-ism, for short) if $ \langle Ax - Ay, x-y \rangle \geq \nu\|Ax-Ay\|^2$ for all $x,y \in D(A)$. \\
\end{itemize}
\end{dfn}
\noindent The following proposition gathers some useful results on the relationship between averaged
operators and inverse strongly monotone operators.

\begin{prop}(\cite{cby})\label{need2}
If $T: \mathbb{R}^n \rightarrow  \mathbb{R}^n$ is an operator, then the following statements hold:
\begin{itemize}
  \item[(a)] $T$ is nonexpansive if and only if the complement $I-T$ is $\frac{1}{2}$-ism;
\item[(b)] If $T$ is $\nu$-ism, then for $\gamma > 0,~~\gamma T$ is $\frac{\nu}{\gamma}$-ism;
\item[(c)] $T$ is averaged if and only if the complement $I-T$ is $\nu$-ism for some $\nu > 1/2$. Indeed, for $\alpha \in (0, 1),~~ T$ is $\alpha$-averaged if and only if $I-T$ is $\frac{1}{2\alpha}$-ism.
\end{itemize}
\end{prop}

\section{The algorithm and convergence analysis}\label{Sec:Method}

In this section, we give a precise statement of our method and its convergence
analysis. We first state the assumptions that will be needed throughout  the rest of this paper.

\begin{asm}\label{Ass:VI} Considering problem \eqref{bilevel1}--\eqref{bilevel2}, let the following hold:
\begin{itemize}
   \item[(a)] $f:\mathbb{R}^n\rightarrow \mathbb{R}$ is convex and continuously differentiable such that its gradient is Lipschitz continuous with constant $L_f$.
   \item[(b)] $g:\mathbb{R}^n\rightarrow \left(-\infty,\infty\right]$ is proper, lower semicontinous and convex.
   \item[(c)] $h:\mathbb{R}^n\rightarrow \mathbb{R}$ is strongly convex with parameter $\sigma>0$ and continuously differentiable such that its gradient is Lipschitz continuous with constant $L_h$.
   \item[(d)] The set $X^*$ of all optimal solutions of problem \eqref{bilevel2} is nonempty.
\end{itemize}
\end{asm}

\begin{asm}\label{Ass:Parameters}
Suppose $ \{ \alpha_n \}_{n=1}^\infty$ is a sequence in (0,1) and $ \{ \epsilon_n \}_{n=1}^\infty$ is a positive sequence satisfying the following conditions:
\begin{itemize}
   \item[(a)] $\lim_{n \to \infty}  \alpha_n =0$ and $\sum_{n = 1}^{\infty}\alpha_n =\infty$.
   \item[(b)] $\epsilon_n=o(\alpha_n)$, i.e., $\lim_{n \to \infty} \frac{\epsilon_n}{\alpha_n} =0$ (e.g., $\epsilon_n=\frac{1}{(n+1)^2}, \alpha_n=\frac{1}{n+1}$).
    \item[(c)]   $\lambda \in \left(0,\frac{2}{L_f}\right)$ and $\gamma \in \left(0,\frac{2}{L_h+\sigma}\right]$.
\end{itemize}
\end{asm}

\begin{rem}
\noindent Note that the stepsize $\lambda$ in Assumption (c) above is chosen in a larger interval than that of \cite{SabachShtern}. Also, our Assumption (a) is weaker than Assumption C of \cite{SabachShtern} since $\{\alpha_n\}$ is not required in our Assumption (a) to satisfy $\lim_{n \to \infty}  \frac{\alpha_{n+1}}{\alpha_n}=1$ as assumed in Assumption C of \cite{SabachShtern}. Take, for example,
$\alpha_n=\frac{1}{\sqrt{n}},$ when $n$ is odd and $\alpha_n=\frac{1}{n},$ when $n$ is even.
We see that $\{\alpha_n\}$ satisfies Assumption (a) but $\frac{\alpha_{n+1}}{\alpha_n} \not \to 1$.
\end{rem}

\noindent
We next give a precise statement of our \emph{inertial Bilevel Gradient Sequential Averaging Method} (iBiG-SAM) as follows.

\begin{algorithm}[H]
\caption{iBiG-SAM}
\label{Alg:AlgL}
\begin{algorithmic}
 \STATE \textbf{Step 0}: Choose sequences  $ \{ \alpha_n \}_{n=1}^\infty$ and $ \{ \epsilon_n \}_{n=1}^\infty$
      such that the conditions in Assumption~\ref{Ass:Parameters} hold.
      Select arbitrary points $x_0, x_1 \in \mathbb{R}^n$ and $\alpha\geq 3$. Set $n:=1$.
 \STATE \textbf{Step 1}: Given the iterates $x_{n-1}$ and $x_n$ (with $n \geq 1$), choose $\theta_n$ such that we have
      $0\leq \theta_n \leq \bar{\theta}_n$ with $\bar{\theta}_n$ defined by
 \begin{equation}\label{thetaDefine}
    \bar{\theta}_n :=
\left\{\begin{array}{ll}
      \min\left\{\frac{n-1}{n+\alpha-1},\;\, \frac{\epsilon_n}{\|x_n-x_{n-1}\|}\right\} &  \mbox{if } \;\; x_n\neq x_{n-1},\\[1ex]
      \frac{n-1}{n+\alpha-1} & {\rm otherwise}.
     \end{array}
      \right.
 \end{equation}
 \STATE \textbf{Step 2}: Proceed with the following computations:
 \begin{equation}\label{e31}
 \left\{\begin{array}{l}
         y_n=x_n+\theta_n(x_n-x_{n-1}),\\
         s_n={\rm prox}_{\lambda g}(y_n-\lambda \nabla f(y_n)),\\
         z_n=y_n-\gamma \nabla h(y_n),\\
         x_{n+1}=\alpha_n z_n+(1-\alpha_n)s_n,~~n \geq 1.
         \end{array}
         \right.
 \end{equation}
\end{algorithmic}
\end{algorithm}

\begin{rem}\label{mami}
Observe that from Assumption~\ref{Ass:Parameters}  and Algorithm~\ref{Alg:AlgL} we have that
$$
\lim_{n \to \infty} \theta_n \|x_n-x_{n-1}\|=0 \;\;\mbox{ and } \;\;
\lim_{n \to \infty} \frac{\theta_n }{\alpha_n}\|x_n-x_{n-1}\| =0.
$$
\end{rem}

\noindent
Also note that Step 1 in our Algorithm \ref{Alg:AlgL} is easily implemented in numerical computation
since the value of $\|x_n-x_{n-1}\|$ is a priori known before choosing $\theta_n$.

\noindent
We are now in the position to discuss the convergence of iBIG-SAM. Let us define
\begin{eqnarray}\label{proxgrad}
T_\lambda:={\rm prox}_{\lambda g}(I-\lambda \nabla f).
\end{eqnarray}
\noindent The next lemma shows that the prox-grad mapping $T_\lambda$ is averaged. This is an improvement over Lemma 1(i) of \cite{SabachShtern}.

\begin{lem}\label{use1}
The prox-grad mapping $T_\lambda$ \eqref{proxgrad} is $\frac{2+\lambda L_f}{4}$-averaged for all
$\lambda \in \left(0,\frac{2}{L_f}\right)$.
\end{lem}

\begin{proof}
Observe that the Lipschitz condition on $\nabla f$ implies that $\nabla f$ is $\frac{1}{L_f}$-ism
(see \cite{baillon}), which then implies that $\lambda \nabla f$ is $\frac{1}{\lambda L_f}$-ism.
 Hence, by Proposition \ref{need2}(c), $I-\lambda \nabla f$ is $(\frac{\lambda L_f}{2})$-averaged.
 Since ${\rm prox}_{\lambda f}$ is firmly nonexpansive and hence $\frac{1}{2}$-averaged, we
see from Proposition \ref{need1}(d) that the composite ${\rm prox}_{\lambda g}(I-\lambda \nabla f)$
is $\frac{2+\lambda L_f}{4}$-averaged for $\lambda \in (0,\frac{2}{L_f})$.
Hence we have that, $T_\lambda={\rm prox}_{\lambda g}(I-\lambda \nabla f)$ is
$\frac{2+\lambda L_f}{4}$-averaged.
Therefore, we can write
\begin{eqnarray}\label{prox}
T_\lambda&=&{\rm prox}_{\lambda g}(I-\lambda \nabla f)= \left( \frac{2-\lambda L_f}{4} \right)I+\left(\frac{2+\lambda L_f}{4} \right )T\\
&=&(1-\beta)I+\beta T\label{prox2},
\end{eqnarray}
where $\beta:=\frac{2+\lambda L_f}{4} \in [a,b] \subset (1/2,1)$ and $T$ is a nonexpansive mapping.
\end{proof}

\noindent Lemma 1(ii) of \cite{SabachShtern} showed the equivalence between the fixed points of  prox-grad mapping $T_\lambda$ \eqref{proxgrad} and optimal solutions of problem \eqref{bilevel2}. That is, $x\in X^*$ if and only if $x \, = \, T_\lambda x$.
This equivalence will be needed in our convergence analysis in this paper.

\begin{lem}(\cite{SabachShtern})\label{use2}
Suppose that Assumption \ref{Ass:VI} (c) holds. Then, the mapping $S_\gamma$, defined by
$
S_\gamma:=I-\gamma \nabla h,
$
 is a contraction for all $\gamma \in \left(0,\frac{2}{L_h+\sigma }\right]$. That is,
$$
\|S_\gamma (x)-S_\gamma (y) \| \leq \eta\|x-y\|,\;\; \forall x,y \in \mathbb{R}^n.
$$
Here, $I$ represents the identity operator and $\eta:=\sqrt{1-\frac{2\gamma\sigma L_h}{\sigma+L_h}}$.
\end{lem}

\noindent By the statements  of Lemma \ref{use1} and Lemma \ref{use2}, we can re-write \eqref{e31} as
\begin{equation}\label{e32}
\left\{\begin{array}{ll}
         y_n=x_n+\theta_n(x_n-x_{n-1}),\\
         x_{n+1}=\alpha_n S_{\gamma}(y_n)+(1-\alpha_n)(1-\beta)y_n+\beta(1-\alpha_n)Ty_n,~~n \geq 1,
         \end{array}
         \right.
\end{equation}
where $T$ is a nonexpansive mapping, $S_{\gamma}$ is a contraction mapping and $\beta:=\frac{2+\lambda L_f}{4}$.\\

\noindent
Before we proceed with the main result of this section, we first show that the iterative sequence generated by our algorithm is bounded.

\begin{lem}\label{c31}
Let Assumptions~\ref{Ass:VI} and \ref{Ass:Parameters} be satisfied.
Then the sequence $\{x_n\}$ generated by Algorithm~\ref{Alg:AlgL} is bounded.
\end{lem}

\begin{proof}
From \eqref{e31}, for any $z \in X^*$, we have $z \in F(T_{\lambda})=F(T)$. Therefore,
\begin{eqnarray}\label{chi1}
\|x_{n+1}-z\| &\leq& \alpha_n\|S_{\gamma}(y_n)-z\|+(1-\alpha_n)(1-\beta)\|y_n-z\| +\beta(1-\alpha_n)\|Ty_n-z\|\nonumber\\
&\leq&\alpha_n\left(\|S_{\gamma}(y_n)-S_{\gamma}(z)\|+\|S_{\gamma}(z)-z\| \right )+(1-\alpha_n)\|y_n-z\|\nonumber \\
&\leq& \alpha_n\|S_{\gamma}(z)-z\|+(1-\alpha_n(1-\eta))\|y_n-z\|\nonumber \\
&\leq&\alpha_n\|S_{\gamma}(z)-z\|+(1-\alpha_n(1-\eta))(\|x_n-z\|+\theta_n\|x_n-x_{n-1}\|) \nonumber \\
&=&(1-\alpha_n(1-\eta))\|x_n-z\|+(1-\alpha_n(1-\eta))\theta_n\|x_n-x_{n-1}\|\nonumber \\
&& \qquad \qquad+\alpha_n\|S_{\gamma}(z)-z\|\nonumber \\
&=& (1-\alpha_n(1-\eta))\|x_n-z\|+\alpha_n(1-\eta)\frac{\|S_{\gamma}(z)-z\|}{1-\eta}\nonumber \\
&& \qquad \qquad+(1-\alpha_n(1-\eta))\theta_n\|x_n-x_{n-1}\|\nonumber \\
&=& (1-\alpha_n(1-\eta))\|x_n-z\|+\alpha_n\Big(\frac{(1-\eta)}{\alpha_n}\frac{\|S_{\gamma}(z)-z\|}{1-\eta}\nonumber \\
&& \qquad \qquad+(1-\alpha_n(1-\eta))\frac{\theta_n}{\alpha_n}\|x_n-x_{n-1}\|\Big).
\end{eqnarray}
Observe that $\sup_{n\geq 1} (1-\alpha_n(1-\eta))\frac{\theta_n}{\alpha_n}\|x_n-x_{n-1}\|$ exists by Remark \ref{mami} and take
$$
M:=\max~\left\{\frac{(1-\eta)}{\alpha_n}\frac{\|S_{\gamma}(z)-z\|}{1-\eta}, \;\;\sup_{n\geq 1} (1-\alpha_n(1-\eta))\frac{\theta_n}{\alpha_n(1-\eta)}\|x_n-x_{n-1}\|\right\}.
$$
\noindent Then \eqref{chi1} becomes
$$
\|x_{n+1}-z\| \leq (1-\alpha_n(1-\eta))\|x_n-z\|+\alpha_nM.
$$
\noindent By Lemma \ref{lm23} , we get that $\{x_n\}$ is bounded. As a consequence, $\{y_n\} $ is also bounded.
\end{proof}

\begin{thm}\label{t31}
Let Assumptions~\ref{Ass:VI} and \ref{Ass:Parameters} hold.
Then the sequence $\{x_n\} $ generated by Algorithm~\ref{Alg:AlgL}
converges to a point $z \in X^*$ satisfying
\begin{eqnarray}\label{e34}
\langle \nabla h(z),x-z\rangle \geq 0 \quad \forall x \in X^*
\end{eqnarray}
and therefore, $z=z_{mn}$ is the optimal solution of problem \eqref{bilevel1}--\eqref{bilevel2}.
\end{thm}

\begin{proof}
Start by observing that
\begin{eqnarray}\label{e6}
\|y_n-z\|^2
&=& \|x_n-z\|^2+2\theta_n\langle x_n-x_{n-1},x_n-z\rangle+\theta_n^2\|x_n-x_{n-1}\|^2.
\end{eqnarray}
From Lemma \ref{lm2} (i) it holds
\begin{eqnarray}\label{e7}
2 \langle x_n-x_{n-1}, x_n-z\rangle= -\|x_{n-1}-z\|^2+\|x_n-z\|^2 + \| x_n-x_{n-1}\|^2.
\end{eqnarray}
Substituting \eqref{e7} into \eqref{e6}, we obtain
 \begin{eqnarray}\label{lara}
\|y_n-z\|^2
&=& \|x_n-z\|^2+\theta_n(-\|x_{n-1}-z\|^2+\|x_n-z\|^2+\| x_n-x_{n-1}\|^2) \nonumber\\
 && \quad +\theta_n^2\|x_n-x_{n-1}\|^2\nonumber \\
 &=& \|x_n-z\|^2+\theta_n(\|x_n-z\|^2-\|x_{n-1}-z\|^2)\nonumber\\
 && \quad +\theta_n(1+\theta_n)\|x_n-x_{n-1}\|^2 \nonumber\\
  &\leq & \|x_n-z\|^2+\theta_n(\|x_n-z\|^2-\|x_{n-1}-z\|^2)\nonumber\\
  && \quad +\;\; 2 \theta_n \|x_n-x_{n-1}\|^2,
 \end{eqnarray}
 where the last inequality follows from the fact that $\theta_n \in [0,1)$.
Using Lemma \ref{lm2} (ii) and (iii), we obtain from \eqref{e31} that
 \begin{eqnarray}\label{laraa}
 \|x_{n+1}-z\|^2&=&\|\alpha_n(S_{\gamma}(y_n)-z)+(1-\alpha_n)(1-\beta)(y_n-z)+\beta(1-\alpha_n)(Ty_n-z)\|^2\nonumber\\
  &\leq & \|(1-\alpha_n)(1-\beta)(y_n-z)+\beta(1-\alpha_n)(Ty_n-z)\|^2 \nonumber \\
  && \quad +\;\; 2\langle \alpha_n(S_{\gamma}(y_n)-z), x_{n+1}-z \rangle \nonumber \\
 &=&(1-\alpha_n)^2(1-\beta)\|y_n-z\|^2+\beta(1-\alpha_n)^2\|Ty_n-z\|^2 \nonumber\\
&& \quad -\;\;\beta(1-\beta)(1-\alpha_n)^2\|y_n-Ty_n\|^2+2\alpha_n\langle S_{\gamma}(y_n)-z, x_{n+1}-z \rangle\nonumber \\
&\leq&(1-\alpha_n)^2\|y_n-z\|^2-\beta(1-\beta)(1-\alpha_n)^2\|y_n-Ty_n\|^2\nonumber \\
&& \quad +\;\;2\alpha_n\langle S_{\gamma}(y_n)-z, x_{n+1}-z \rangle.
\end{eqnarray}
Combining \eqref{lara} and \eqref{laraa}, we get
 \begin{eqnarray}\label{lara2}
\|x_{n+1}-z\|^2
&\leq&  (1-\alpha_n)^2 \|x_n-z\|^2-\beta(1-\beta)(1-\alpha_n)^2\|y_n-Ty_n\|^2\nonumber \\
&& \quad +\;\;\theta_n(1-\alpha_n)^2(\|x_n-z\|^2-\|x_{n-1}-z\|^2)\nonumber\\
&&\quad +\;\;2\theta_n(1-\alpha_n)^2\|x_n-x_{n-1}\|^2 \nonumber\\
&& \quad +\;\; 2\alpha_n\left\langle S_{\gamma}(y_n)-z, x_{n+1}-z \right\rangle.
 \end{eqnarray}
Setting
$
\Gamma_n:=\|x_n-z\|^2
$
for all $n \geq 1$, it follows from  \eqref{lara2} that
\begin{eqnarray}\label{dubem}
\Gamma_{n+1}
&\leq&(1-\alpha_n)^2\Gamma_n-\beta(1-\beta)(1-\alpha_n)^2\|y_n-Ty_n\|^2
+\theta_n(1-\alpha_n)^2(\Gamma_n-\Gamma_{n-1})\nonumber\\
&& \quad +\;\;2\theta_n(1-\alpha_n)^2\|x_n-x_{n-1}\|^2+2\alpha_n\langle S_{\gamma}(y_n)-z, x_{n+1}-z \rangle.
\end{eqnarray}

\noindent We consider two cases for the rest of the proof.\\

\noindent {\bf Case 1:}
Suppose there exists a natural number $n_0$ such that $\Gamma_{n+1}\leq\Gamma_n$ for all
$n \geq n_0$. Therefore, $\lim_{n \to \infty} \Gamma_n$ exists.
From \eqref{dubem}, we have
\begin{eqnarray}\label{eq6}
&& \beta(1-\beta)(1-\alpha_n)^2\|y_n-Ty_n\|^2\qquad \qquad \nonumber\\
&\leq& (\Gamma_n-\Gamma_{n+1})+\theta_n(1-\alpha_n)^2 (\Gamma_n-\Gamma_{n-1})\nonumber\\
&&  + \;\; 2\theta_n(1-\alpha_n)^2\|x_n-x_{n-1}\|^2+2\alpha_n\langle S_{\gamma}(y_n)-z, x_{n+1}-z \rangle.
\end{eqnarray}
Using Assumption~\ref{Ass:Parameters} (noting that $\lim_{n \to \infty} \theta_n \|x_n-x_{n-1}\|=0$ and
$\{x_n\}$, $\{y_n\}$ are bounded), we have
$$
\lim_{n \to \infty}\beta(1-\beta)(1-\alpha_n)^2\|y_n-Ty_n\|=0.
$$
\noindent Observe that $ \underset{n\rightarrow \infty}\liminf\beta(1-\beta)(1-\alpha_n)^2=\underset{n\rightarrow \infty}\lim\beta(1-\beta)(1-\alpha_n)^2=\beta(1-\beta) >0$ and this immediately implies that
$$
\lim_{n \to \infty}\|Ty_n-y_n\|=0.
$$
\noindent Since $\{x_n\}$ is bounded, take a subsequence $\{x_{n_k}\}$ of $\{x_n\}$ such that
$x_{n_k}\rightarrow p \in \mathbb{R}^n$  and using the definition of contraction mapping $S_{\gamma}$
 in Lemma \ref{use2}, we have
\begin{eqnarray}\label{eq7}
\limsup_{n \to \infty} \langle S_{\gamma}(z)-z, x_n-z \rangle &=&\lim_{k \to \infty} \langle S_{\gamma}(z)-z, x_{n_k}-z \rangle \nonumber \\
&=& \langle S_{\gamma}(z)-z, p-z \rangle=\langle \nabla h(z),z-p\rangle.
\end{eqnarray}
From $y_n=x_n+\theta_n(x_n-x_{n-1})$, we get
$$
\|y_n-x_n\|=\theta_n\|x_n-x_{n-1}\|\rightarrow 0.
$$
\noindent
Since $x_{n_k}\rightarrow p $, we have $y_{n_k}\rightarrow p $.
Lemma \ref{lm25} then guarantees that $p \in F(T)=X^*$.
Furthermore, we have from \eqref{e34} and \eqref{eq7} that
\begin{eqnarray}\label{ayi}
\limsup_{n \to \infty}~\langle S_{\gamma}(z)-z, x_n-z \rangle \leq 0.
\end{eqnarray}
\noindent
From the contraction of $S_{\gamma}$ and \eqref{lara}, we can write
\begin{eqnarray*}
2\alpha_n\langle S_{\gamma}(y_n)-z, x_{n+1}-z \rangle
&=& 2\alpha_n\langle S_{\gamma}(y_n)-S_{\gamma}(z)+S_{\gamma}(z)-z, x_{n+1}-z \rangle \nonumber\\
&\leq& 2\alpha_n\eta\|y_n-z\|\|x_{n+1}-z\|+2\alpha_n\langle S_{\gamma}(z)-z, x_{n+1}-z \rangle\nonumber\\
&\leq& \alpha_n\eta(\|y_n-z\|^2+\|x_{n+1}-z\|^2)+2\alpha_n\langle S_{\gamma}(z)-z, x_{n+1}-z \rangle\nonumber\\
&\leq& \alpha_n\eta(\Gamma_n+\theta_n(\Gamma_n-\Gamma_{n-1})+2\theta_n\|x_n-x_{n-1}\|^2)\\
&&+\;\;2\alpha_n\langle S_{\gamma}(z)-z, x_{n+1}-z \rangle+\alpha_n\eta\|x_{n+1}-z\|^2.
\end{eqnarray*}
Therefore from \eqref{dubem} it holds
\begin{eqnarray}
\Gamma_{n+1}&\leq&(1-\alpha_n)^2\Gamma_n+\theta_n(1-\alpha_n)^2(\Gamma_n-\Gamma_{n-1})\nonumber\\
&& \quad +\;\;2\theta_n(1-\alpha_n)^2\|x_n-x_{n-1}\|^2+2\alpha_n\langle S_{\gamma}(y_n)-z, x_{n+1}-z \rangle \nonumber\\
&\leq& ((1-\alpha_n)^2+\alpha_n\eta)\Gamma_n +\theta_n((1-\alpha_n)^2+\alpha_n\eta)(\Gamma_n-\Gamma_{n-1})\nonumber\\
&&\quad +\;\;2\theta_n((1-\alpha_n)^2+\alpha_n\eta)\|x_n-x_{n-1}\|^2+2\alpha_n\langle S_{\gamma}(z)-z, x_{n+1}-z \rangle
\nonumber\\
&&\quad +\;\;\alpha_n\eta\|x_{n+1}-z\|^2\nonumber\\
&\leq& ((1-\alpha_n)^2+\alpha_n\eta)\Gamma_n+ \theta_n((1-\alpha_n)^2+\alpha_n\eta)\|x_n-x_{n-1}\|(\sqrt{\Gamma_n}+\sqrt{\Gamma_{n-1}})\nonumber\\
&& \quad +\;\;2\theta_n((1-\alpha_n)^2+\alpha_n\eta)\|x_n-x_{n-1}\|^2+2\alpha_n\langle S_{\gamma}(z)-z, x_{n+1}-z \rangle \nonumber\\
&&\quad +\;\;\alpha_n\eta\|x_{n+1}-z\|^2\nonumber\\
&=&((1-\alpha_n)^2+\alpha_n\eta)\Gamma_n+ \theta_n\|x_n-x_{n-1}\|M_2+\alpha_n\eta\|x_{n+1}-z\|^2\nonumber\\
&& \quad +\;\;2\alpha_n\langle S_{\gamma}(z)-z, x_{n+1}-z \rangle,
\end{eqnarray}
where
$$
M_2:=\sup_{n \geq 1} \left((1-\alpha_n)^2+\alpha_n\eta)(\sqrt{\Gamma_n}+\sqrt{\Gamma_{n-1}}+2((1-\alpha_n)^2+\alpha_n\eta)\|x_n-x_{n-1}\|\right).
$$
Therefore
\begin{eqnarray}\label{eq9}
\Gamma_{n+1}&\leq& \frac{(1-\alpha_n)^2+\alpha_n\eta}{1-\alpha_n \eta}\Gamma_n+ \frac{\theta_n\|x_n-x_{n-1}\|M_2}{1-\alpha_n \eta}\nonumber \\
&& \quad +\;\;2\frac{\alpha_n}{1-\alpha_n \eta}\langle S_{\gamma}(z)-z, x_{n+1}-z \rangle \nonumber \\
&\leq& \left(1-\frac{2(1-\eta)\alpha_n}{1-\alpha_n \eta} \right)\Gamma_n+\frac{\theta_n\|x_n-x_{n-1}\|M_2}{1-\alpha_n \eta}\nonumber \\
&& \quad +\;\;2\frac{\alpha_n}{1-\alpha_n \eta}\langle S_{\gamma}(z)-z, x_{n+1}-z \rangle +\frac{\alpha_n^2}{1-\alpha_n \eta}\Gamma_n\nonumber \\
&\leq& \left(1-\frac{2(1-\eta)\alpha_n}{1-\alpha_n \eta} \right)\Gamma_n \nonumber \\
&& +\;\;\frac{2(1-\eta)\alpha_n}{1-\alpha_n \eta}\Big\{
\frac{\theta_n\|x_n-x_{n-1}\|M_2}{2(1-\eta)\alpha_n} +\frac{\alpha_n \Gamma_{n_0}}{2(1-\eta)}+\frac{1}{1-\eta}\langle S_{\gamma}(z)-z, x_{n+1}-z \rangle \Big\}\nonumber \\
&=&(1-\delta_n)\Gamma_n+\delta_n \sigma_n,
\end{eqnarray}
where $\delta_n:=\frac{2(1-\eta)\alpha_n}{1-\alpha_n \eta} \,$  and
$$
\sigma_n:=\frac{\theta_n\|x_n-x_{n-1}\|M_2}{2(1-\eta)\alpha_n}+\frac{\alpha_n \Gamma_{n_0}}{2(1-\eta)}+\frac{1}{1-\eta}\langle S_{\gamma}(z)-z, x_{n+1}-z \rangle.
$$
Using Lemma \ref{lm23} (ii) and Assumption~~\ref{Ass:Parameters} in \eqref{eq9}, we get
$\Gamma_n=\|x_n-z\|\rightarrow 0$ and thus $x_n\rightarrow z$ as $n\rightarrow \infty$.\\

\noindent {\bf Case 2:}  Assume that there is no $ n_0 \in \mathbb{N} $ such that
$\{\Gamma_{n}\}_{n=n_0}^\infty$ is monotonically decreasing.
Let
$\tau :\mathbb{N}\rightarrow \mathbb{N}$ be a mapping defined for all
$n\geq n_{0}$ (for some $n_{0}$ large enough) by
$$
   \tau(n):=\max~\left\{k\in \mathbb{N}:k\leq n,\Gamma_{k}\leq \Gamma_{k+1}\right\},
$$
i.e.\ $ \tau (n) $ is the largest number $ k $ in $ \{ 1, \ldots, n \} $ such
that $ \Gamma_k $ increases at $ k = \tau (n) $; note that, in view of
Case 2, this $ \tau (n) $ is well-defined for all sufficiently large $ n $.
Clearly, $\tau$ is a non-decreasing sequence \cite{MaingeSIAM} such that $\tau(n)\rightarrow
\infty$ as $n \rightarrow \infty$ and
$$
   0\leq\Gamma_{\tau(n)}\leq\Gamma_{\tau(n)+1}, \quad \forall n\geq n_{0}.
$$
Using similar techniques as in \eqref{eq6}, it is easy to
show that
$$
\lim_{n \to \infty}\|Ty_{\tau(n)}-y_{\tau(n)}\|
=\lim_{n \to \infty}\|y_{\tau(n)}-x_{\tau(n)}\|
=\lim_{n \to \infty}\|Ty_{\tau(n)}-x_{\tau(n)}\|=0.
$$
Furthermore, using the boundedness of $ \{ x_n \} $, $ \{ y_n \} $ and
Assumption \ref{Ass:Parameters}, we get
\begin{eqnarray}\label{s7}
   \|x_{\tau(n)+1}-x_{\tau(n)}\| &\leq&\alpha_{\tau(n)}\|S_{\gamma}\Big(y_{\tau(n)}\Big)-x_{\tau(n)}\|+
   \theta_{\tau(n)}\|y_{\tau(n)}-x_{\tau(n)}\| \nonumber \\
   & & +\;\;(1-\alpha_{\tau(n)})\|Ty_{\tau(n)}-x_{\tau(n)}\|\longrightarrow  0 \, \mbox{ as } \, n \rightarrow \infty.
\end{eqnarray}
Since $\{x_{\tau(n)}\}$ is bounded, there exists a subsequence of
$\{x_{\tau(n)}\}$, still denoted by $\{x_{\tau(n)}\}$, which converges
to some $p\in F(T)$. Similarly, as in Case 1 above, we can show that we have
$\underset{n\rightarrow\infty}\limsup~\left\langle S_\gamma(z)-z, x_{\tau(n)+1}-z\right\rangle
   \leq 0.$
Following \eqref{eq9}, we obtain
\begin{eqnarray}\label{flasche}
\Gamma_{\tau(n)+1}&=&(1-\delta_{\tau(n)})\Gamma_{\tau(n)}+\delta_{\tau(n)} \sigma_{\tau(n)},
\end{eqnarray}
which implies that
$
\|x_{\tau(n)}-z\|^2 \leq \sigma_{\tau(n)}
$
while noting that $\Gamma_{\tau(n)} \leq \Gamma_{\tau(n)+1}$ and $\alpha_{\tau(n)}>0$ hold. This leads to
$
\underset{n\rightarrow\infty}\limsup~\|x_{\tau(n)}-z\|^2 \leq 0.
$
Thus, we have
$$
\lim_{n \to \infty}\|x_{\tau(n)}-z\|=\lim_{n \to \infty} \Gamma_{\tau(n)}=0,
$$
which in turn implies $\underset{n\rightarrow\infty}\lim \|x_{\tau(n)+1}-z\|=0$.
Furthermore, for $n\geq n_{0}$, it is easy to see that $\Gamma_n \leq
\Gamma_{\tau(n)+1}$ (observe that $\tau(n)\leq n$ for $n\geq n_{0}$ and consider the three cases:
$\tau(n)=n$, $\tau(n)=n-1$ and $\tau(n)<n-1$. For the first and second cases, it is obvious that $\Gamma_n \leq
\Gamma_{\tau(n)+1}$ for $n\geq n_{0}$. For the third case $\tau(n)\leq n-2$, we have from the definition of $\tau(n)$ and
for any integer $n\geq n_{0}$ that $\Gamma_{j} \geq \Gamma_{j+1}$ for $\tau(n)+1\leq j\leq n-1.$ Thus,
$\Gamma_{\tau(n)+1} \geq \Gamma_{\tau(n)+2}\geq \cdots\geq \Gamma_{n-1}\geq \Gamma_n$). As a consequence,
we obtain for all sufficiently large $ n $ that
$ 0\leq \Gamma_n \leq \Gamma_{\tau(n)+1} $.
Hence $\underset{n\rightarrow\infty} \lim\Gamma_n=0$. Therefore, $\{x_n\}$
converges to $z$.
\end{proof}

\begin{rem}
Suppose that Assumption \ref{Ass:VI}(c)
is replaced with the following milder condition:
``$h:\mathbb{R}^n \rightarrow \mathbb{R}$ is strongly convex with parameter $\sigma>0$ and $L_h$-Lipschitz continuous''. Then the step involving $z_n$ in Algorithm \ref{Alg:AlgL} can be replaced by
\begin{eqnarray*}
z_n&=&y_n-\gamma \nabla M_{\gamma h}(y_n)\\
&=&y_n-\gamma\frac{1}{\gamma}(y_n-{\rm prox}_{\gamma h}(y_n))\\
&=& {\rm prox}_{\gamma h}(y_n),
\end{eqnarray*}
where $M_{\gamma h}$ is the Moreau envelop of $h$, defined by
$$
M_{\gamma h}(x):=\min_{u \in \mathbb{R}^n}\Big\{h(u)+\frac{1}{2\gamma}\|u-x\|^2 \Big\},
$$
 which is continuously differentiable (see \cite{Bauschkebook}) with $\nabla M_{\gamma h}(x)=\frac{1}{\gamma}(x-{\rm prox}_{\gamma h}(x))$ and global convergence is still obtained as in Theorem \ref{t31} using Lemma 6 of \cite{SabachShtern}.
\end{rem}

\noindent
We give some brief comments on the nonasymptotic $O(1/n^2)$ convergence rate of some estimates obtained in Theorem \ref{t31}.
\begin{rem}
Observe that for Algorithm \ref{Alg:AlgL}, we have
$\theta_n\|x_n-x_{n-1}\| \leq \epsilon_n$ for all $n \geq 1$.
\noindent
If we choose $\epsilon_n:=\frac{c}{n^2}$, where $c>0$, then
$\theta_n\|x_n-x_{n-1}\| \leq \frac{c}{n^2}$ for all $n \geq 1$.
\noindent
 Thus,
$ \theta_n\|x_n-x_{n-1}\| = O(1/n^2)$
\noindent
and consequently
$$ \|y_n-x_n\| =\theta_n\|x_n-x_{n-1}\| = O(1/n^2).$$
\noindent
Full details on the convergence rate of the result in Theorem \ref{t31} is left for further careful investigation in a separate work.
\end{rem}

\section{Numerical Results}\label{Sec:Numerics}
For numerical implementation of our proposed method in Section \ref{Sec:Method} we consider the inverse problems
tested in \cite{SabachShtern} and give numerical comparison with the proposed Algorithm \ref{Alg:AlgL} (iBiG-SAM)
and that of BiG-SAM method in \cite{SabachShtern}. The codes are implemented in Matlab.
We perform all computations on a windows desktop with an Intel(R) Core(TM) i7-2600 CPU at 3.4GHz and 8.00 GB of memory.
We take $\alpha_n=\frac{2\kappa}{n(1-\beta)}$ with $\kappa=0.1$, which is the best choice for BiG-SAM considered in \cite{SabachShtern} and $\beta \in [0,1)$ defined as in \eqref{prox2} and $\theta_n=\bar{\theta}_n$ as in \eqref{thetaDefine} with $\alpha = 3$ and $\epsilon_n=\alpha_n/n^{0.01}$ for iBiG-SAM.
\begin{table}
\caption{Averaged over 100 runs for each problem in Example \ref{example1}}\label{Table1}
\bigskip
\centering
\renewcommand{\arraystretch}{1.25}
\begin{tabular}{c | c c | c c }
\hline
~&iBiG-SAM&&BiG-SAM&~\\
Problem &Number of iterations&time(sec.)&Number of iterations&time(sec.)\\
\hline
Baart &119.15&1.7253&145.67&2.1089\\
\hline
Foxgood &122.04&1.7861&149.78&2.1885\\
\hline
Phillips&120.77&1.7463&148.18&2.1397\\
\hline
\end{tabular}
\end{table}
\begin{exm}\label{example1}
\noindent Following \cite{SabachShtern}, the inner objective function is taking as
$$
\varphi(x):=\frac{1}{2} \|Ax-b\|^2 +\delta_X(x),
$$
\noindent where $\delta_X$ is the indicator function over the nonnegative orthant $X:=\{x \in \mathbb{R}^n: x \geq 0\}$. Furthermore, we take the outer objective function as
\begin{equation}\label{Donali}
h(x) :=\frac{1}{2}x^TQx,
\end{equation}
\noindent where $Q$ is a positive definite matrix.
It is clear that $L_f=\|A^tA\|$ and $L_h=\|Q\|$. We choose $\lambda=\frac{1}{L_f}$ and  $\gamma=\frac{2}{L_h+\sigma}$. \\

\noindent
Following \cite{Beck}, we consider three inverse problems, i.e., \emph{Baart}, \emph{Foxgood}, and \emph{Phillips} \cite{SabachShtern}. For each of these problems, we generated the corresponding $1,000$ by $1,000$
exact linear system of the form $Ax = b$, by applying the relevant function (baart, foxgood, and phillips). We then performed the simulation by adding normally
distributed noise with zero mean to the right-hand-side vector $b$, with deviation  $\rho =0.01$.
The matrix $Q$ is defined by $Q =LL'+I$, where $L$ is generated by the function get-l(1,000,1) from the \emph{regularization
tools} (see \url{http://www.imm.dtu.dk/~pcha/Regutools/}) and approximates the first-derivative operator.
\begin{figure}[H]
\begin{centering}
\includegraphics[width=0.5\textwidth]{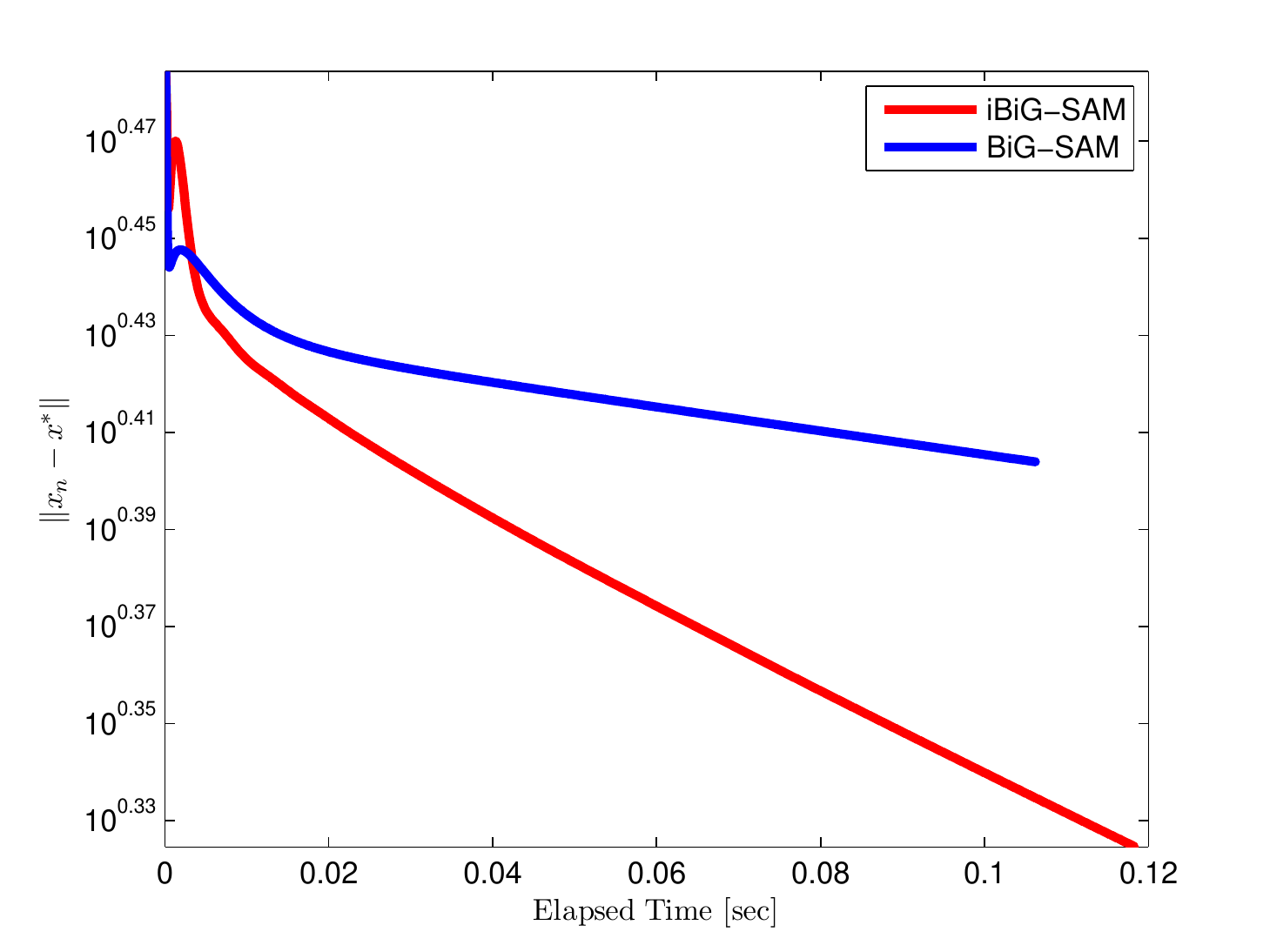}\includegraphics[width=0.5\textwidth]{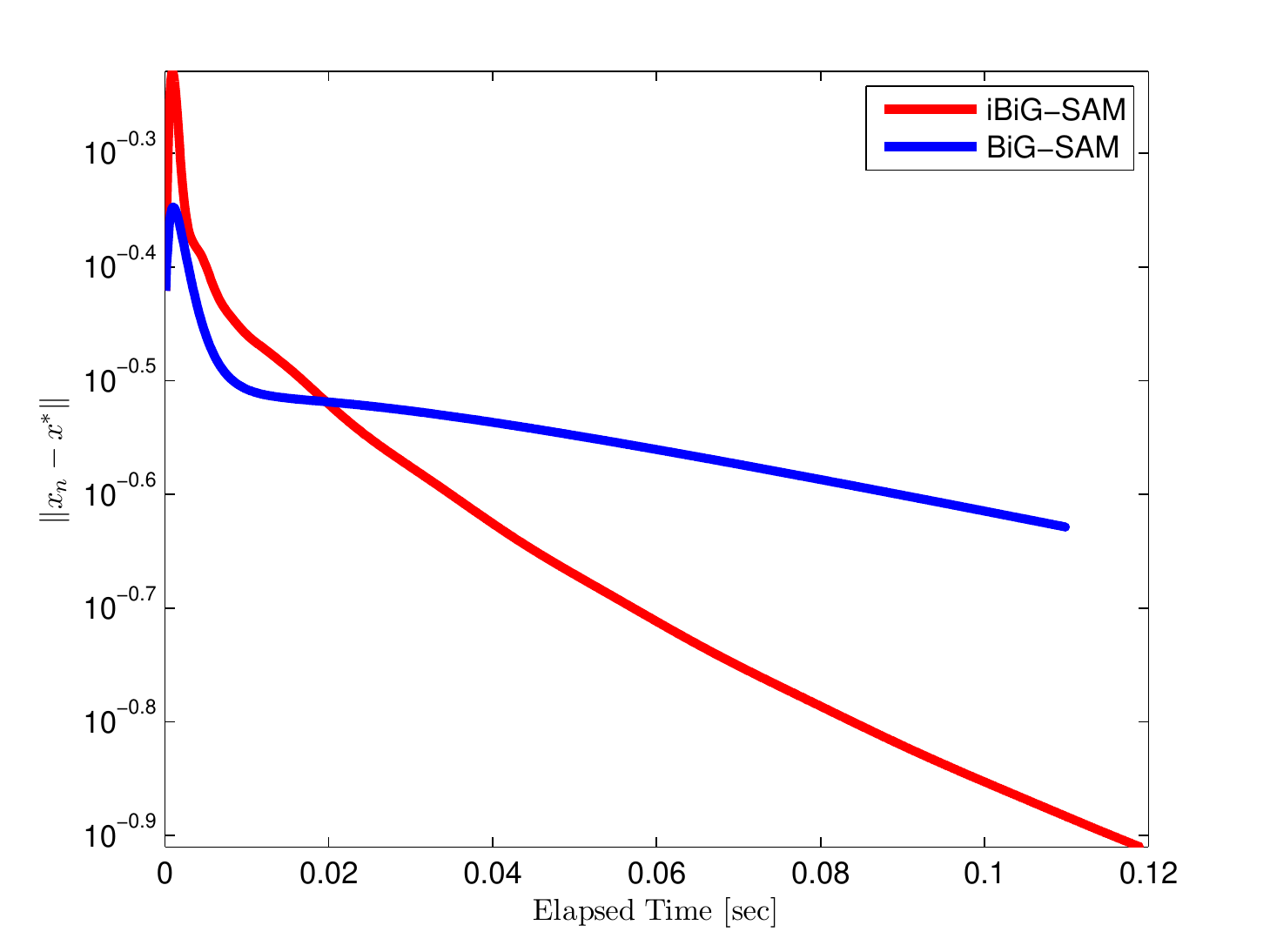}
\par\end{centering}
\protect\caption{Distance to optimal solution v.s. CPU time for problems Baart (left) and Foxgood (right) with $n=100$.} \label{Figure0}
\end{figure}

\noindent Following \cite{SabachShtern}, we use the stopping condition $(\varphi(x_n)-\varphi^* )/\varphi^* \leq 10^{-2}$ for both methods, where $\varphi^*$ is the optimal value of the inner problem computed in advance by BiG-SAM with $1000$ iterations.
In Table \ref{Table1} we present the averaged number of iterations and time (out of 100 runs) until the algorithms reach the stopping criterion.
It can be seen that iBiG-SAM outperforms BiG-SAM (on averaged about 20\%) in all problems tested.\\

\noindent In Figure \ref{Figure0}, we compare the behavior of iBiG-SAM wih BiG-SAM for Baart and Foxgood problems when $n=100$. \qed
\end{exm}

\begin{table}[htp]
\caption{Averaged over 100 runs for each problem in Example \ref{example2}}\label{Table2}
\bigskip
\centering
\renewcommand{\arraystretch}{1.25}
\begin{tabular}{c | c c | c c }
\hline
~&iBiG-SAM&&BiG-SAM&~\\
Parameters &Iterations&time (sec.)&Iterations&time (sec.)\\
\hline
$\alpha=3, m = 100, n = 500$ &43.32&0.0498&60.43&0.0697\\
\hline
$\alpha=4, m = 200, n = 500$&12.25&0.017&18.65&0.0252\\
\hline
$\alpha=5, m = 500, n = 1000$&12.31&0.124&18.07&0.1793\\
\hline
\end{tabular}
\end{table}

\begin{exm}\label{example2}
We now look at the case when $g$ is not an indicator function. In this case, the methods proposed in \cite{Beck,FerrisMangasarian,Solodov2} cannot be applied. We still give a comparison of our method with BiG-SAM \eqref{e1}. The inner objective function is taking here as
$$
\varphi(x):=\frac{1}{2}\|Ax-b\|_{2}^{2}+\mu\|x\|_{1},
$$
where $A\in \mathbb{R}^{m\times n}$ is a given matrix, $b$ is a given vector and $\mu$ a positive scalar. This is LASSO (Least Absolute selection and Shrinkage Operator) \cite{Tibshirami} in compressed sensing. The proximal map with $g(x)=\mu\|x\|_{1}$ is given as ${\rm prox}_{g}(x)=\arg\min_{u}\mu\|x\|_{1}+\dfrac{1}{2}\|u-x\|_{2}^{2},$ which is separable in indices. Thus, for $x\in \mathbb{R}^n$,
   \begin{eqnarray*}
{\rm prox}_{g}(x)\,=\,\left({\rm prox}_{\mu|.|_{1}}(x_1),\ldots,{\rm prox}_{\mu|.|_{1}}(x_n) \right) = \left(\beta_1,\ldots,\beta_n\right),
\end{eqnarray*}
where $\beta_k={\rm sgn}(x_k)\max\{|x_k|-\mu,0\}$ for $k=1,2,\ldots,n$.
As in Example \ref{example1}, we take the outer objective function as in \eqref{Donali}
with $Q$ similarly being a positive definite matrix.
\begin{figure}[htp]
\begin{centering}
\includegraphics[width=0.5\textwidth]{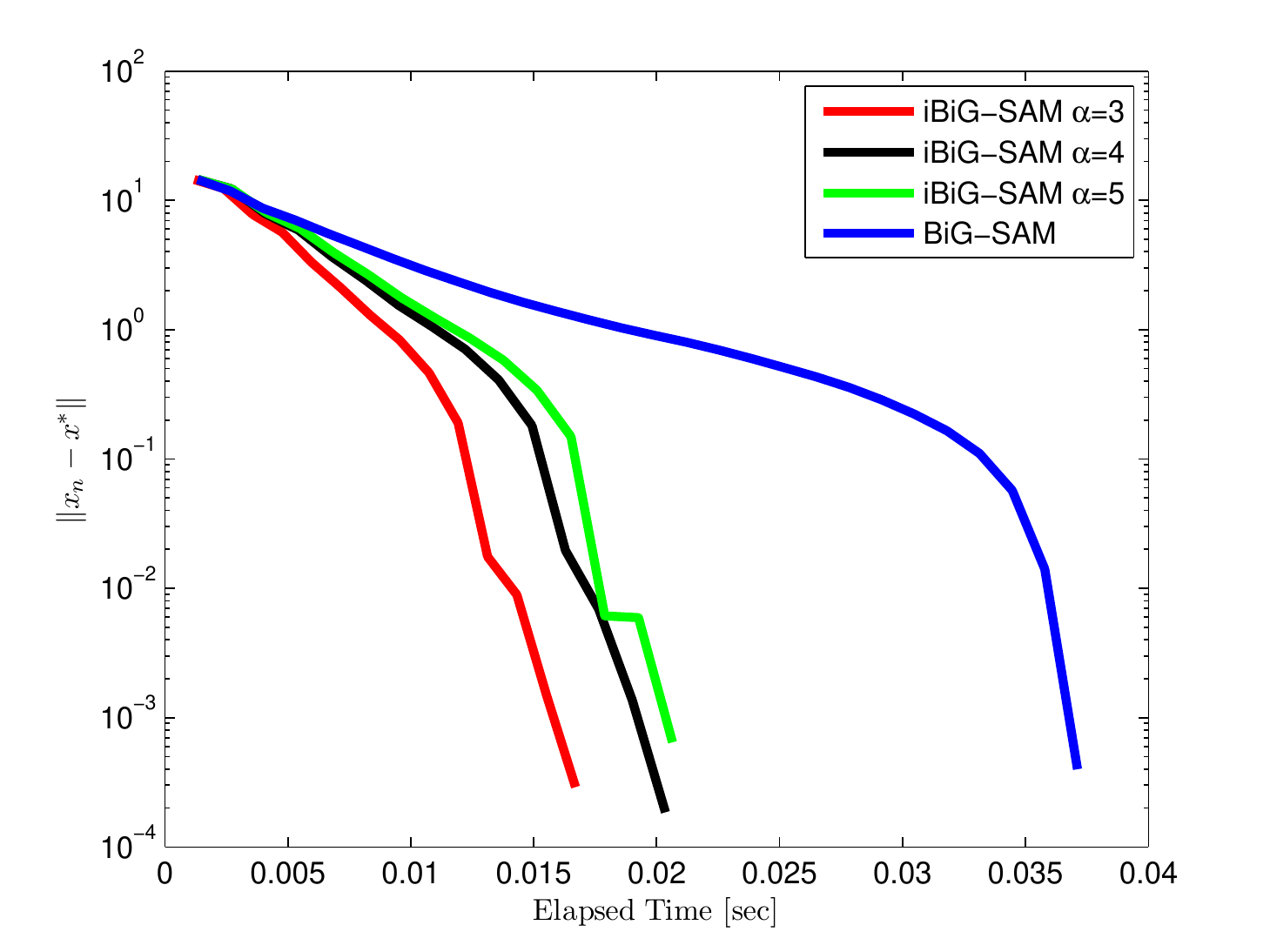}\includegraphics[width=0.5\textwidth]{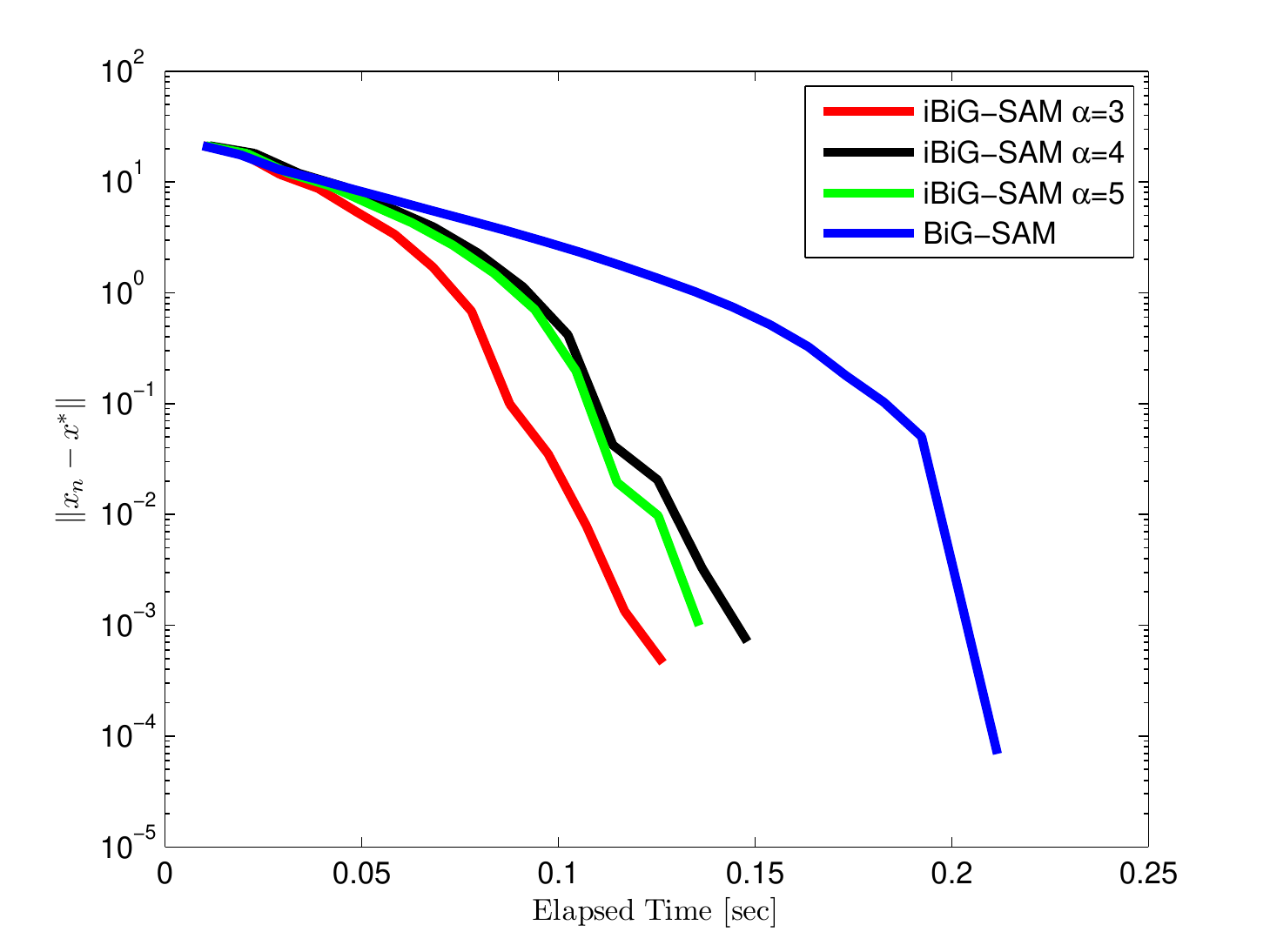}
\par\end{centering}
\protect\caption{Distance to optimal solution v.s. CPU time when $m=100,n=500$ (left) and $m=500,n=1000$ (right)} \label{Figure1}
\end{figure}
We take $\mu=0.5$, and the data $b$ is generated as $Ax + \delta e$, where $A$ and $e$ are random matrices whose elements
are normally distributed with zero mean and variance $1$, and $\delta = 0.01$, and $x$ is a generated
sparse vector. The stopping condition is $\|x_n-x^*\|\leq \epsilon$ with $\epsilon=10^{-3}$ and $x^*$ computed in advance by
BiG-SAM with $1000$ iterations. In Table \ref{Table2} we present the averaged number of iterations and time (out of 100 runs) until the algorithms reach the stopping criterion for different choices of $\alpha \geq 3$ in different dimensional spaces. Again iBiG-SAM outperforms BiG-SAM in all simulations.\\

\noindent In Figure \ref{Figure1}, we compare the behavior of BiG-SAM wih iBiG-SAM for different parameters $\alpha$. It seems that iBiG-SAM with $\alpha=3$ takes advantage over other values tested. \qed
\end{exm}

\noindent Interested readers can download the codes used for the experiments above via the following link (under \emph{iBIG-SAM}), in order to proceed with their own tests on other scenarios  of Examples \ref{example1} and \ref{example2} or to use corresponding adjustments for calculations on new examples:   \url{http://www.southampton.ac.uk/~abz1e14/solvers.html}
\section{Concluding Remarks}\label{Sec:Final}
The paper has introduced and proved the global convergence of an inertial extrapolation-type method for solving simple convex bilevel optimization problems in finite dimensional Euclidean spaces. One can easily check that the results developed here remain valid in infinite dimensional Hilbert spaces.
Based on the numerical experiments conducted, we illustrated that our method outperforms the best known algorithm recently proposed in \cite{SabachShtern} to solve problems of the form  \eqref{bilevel1}--\eqref{bilevel2}.  Our next project in this subject area is to derive the convergence rate of the method proposed in this paper.

\renewcommand{\baselinestretch}{0.95}
{\small{

}}
\end{document}

%% file: structure.tex
\usepackage[
nochapters, 
beramono, 
eulermath,
pdfspacing, 
dottedtoc 
]{classicthesis} 

\usepackage{arsclassica} 

\usepackage[T1]{fontenc} 

\usepackage[utf8]{inputenc} 

\usepackage{graphicx} 
\graphicspath{{Figures/}} 

\usepackage{enumitem} 

\usepackage{lipsum} 

\usepackage{subfig} 

\usepackage{amsmath,amssymb,amsthm} 

\usepackage{varioref} 

\theoremstyle{definition} 

\theoremstyle{plain} 

\theoremstyle{remark} 

\hypersetup{
colorlinks=true, breaklinks=true, bookmarks=true,bookmarksnumbered,
urlcolor=webbrown, linkcolor=RoyalBlue, citecolor=webgreen, 
pdftitle={}, 
pdfauthor={\textcopyright}, 
pdfsubject={}, 
pdfkeywords={}, 
pdfcreator={pdfLaTeX}, 
pdfproducer={LaTeX with hyperref and ClassicThesis} 
}